\documentclass[11pt, reqno]{amsart}
\usepackage[margin=2.5cm]{geometry}
\usepackage[utf8]{inputenc}
\usepackage[T1]{fontenc}

\usepackage{amsmath,amsthm,amssymb,amsfonts}
\usepackage{graphicx,subfig}
\usepackage[english]{babel}
\usepackage{xcolor}
\usepackage{hyperref}

\usepackage{mathtools}
\usepackage{esint}
\usepackage{endnotes}
\usepackage{color}

\newcommand{\B}{\mathbb{B}}

\newcommand{\s}{\mathbb{S}}

\newtheorem{theoreme}{Theorem}[section]

\newtheorem{proposition}{Proposition}[section]

\newtheorem{exemple}{Example}[section]
\newtheorem{remarque}{Remark}[section]

\title{Upper bound for Steklov eigenvalues of warped products with fiber of dimension 2}
\author[Brisson]{Jade Brisson}
\address{Institut de Math\'ematiques, Universit\'e de Neuch\^atel, Rue Emile-Argand 11, 2000 Neuch\^atel, Suisse}
\email{jade.brisson@unine.ch}

\author[Colbois]{Bruno Colbois}
\address{Institut de Math\'ematiques, Universit\'e de Neuch\^atel, Rue Emile-Argand 11, 2000 Neuch\^atel, Suisse}
\email{bruno.colbois@unine.ch}

\begin{document}
\begin{abstract} In this note, we investigate the Steklov spectrum of the warped product $[0,L]\times_h \Sigma$ equipped with the metric $dt^2+h(t)^2g_\Sigma$, where $\Sigma$ is a compact surface. We find sharp upper bounds for the Steklov eigenvalues in terms of the eigenvalues of the Laplacian on $\Sigma$. We apply our method to the case of metric of revolution on the 3-dimensional ball and we obtain a sharp estimate on the spectral gap between two consecutive Steklov eigenvalues.
  \end{abstract}

\maketitle

\section{Introduction}

Let $(\Sigma,g_\Sigma)$ be a closed connected Riemannian manifold of dimension $m$ equipped with Riemannian metric $g_\Sigma$. Consider the Riemannian manifold $M:=\lbrack 0,L\rbrack\times_h \Sigma$ equipped with the Riemannian metric
\[g(t,p):=dt^2+h(t)^2g_\Sigma\,,\]
where $h$ satisfies
    \begin{itemize}
    \item[(H)] $h$ is a positive smooth function such that $h(0)=h(L)=1$.
\end{itemize}
The Steklov problem on $M$ is defined as the following problem
$$
\begin{cases}
    \Delta u=0\,,&\mbox{ in $M$,}\\
    \partial_\nu u=\sigma u\,,&\mbox{ on $\partial M$.}
\end{cases}
$$
The Steklov eigenvalues form an increasing sequence of positive real numbers
\[0=\sigma_0(h)<\sigma_1(h)\leq\sigma_2(h)\leq\cdots\nearrow+\infty\]
where each eigenvalue is repeated according to its multiplicity.

When $m\geq3$, it is proven in \cite[Theorem 3.5]{cesg2019} that $\sup\{\sigma_1(h):h\text{ satisfies }(H)\}=+\infty$ by constructing a family of functions $(h_\varepsilon)$ verifying $(H)$ for which $\sigma_1(h_\varepsilon)$ is arbitrarily large. However, the authors point out, see \cite[Remark 3.6]{cesg2019}, that such a construction does not exist when $m=2$, because for all functions $h$ satisfying $(H)$, we have $\sigma_1(h)\leq \frac{L\lambda_1}{2}$. Naturally, we can ask what is the value of $\sup\{\sigma_1(h):h \text{ satisfies }(H)\}$?  

In the first, and main part, of this paper, we study upper bounds for the Steklov eigenvalues of $M=[0,L]\times_h\Sigma$ in the case where $\dim \Sigma =2$. We suppose that the Riemannian metric $g_{\Sigma}$ of $\Sigma$ is normalized in such a way that the area $\vert \Sigma  \vert$ of $(\Sigma,g_{\Sigma})$ is one. In the last part of this paper, we look at the case where $\Sigma$ is the round sphere $\mathbb S^2$ of area $1$ and we consider revolution-type metrics on the $3$-dimensional ball.

\medskip
In Section \ref{sec:proofs}, we prove the following upper bound for $\sigma_j(h)$.
\begin{theoreme}\label{thm:upperbound}
    Let $M=\lbrack0,L\rbrack\times_h \Sigma$ be a Riemannian manifold equipped with the metric $g=dt^2+h(t)^2g_\Sigma$, where $(\Sigma,g_\Sigma)$ is a closed Riemannian manifold of dimension 2 such that $|\Sigma|=1$ and $h$ is a smooth function on $\lbrack0,L\rbrack$ such that $h(0)=h(L)=1$. Then, for all $j\geq1$, we have
    \[\sigma_j(h)<\frac{L\lambda_j}{2}\,.\]
\end{theoreme}
Moreover, this upper bound is optimal.

\begin{theoreme}\label{thm:supev}
     Let $M=\lbrack0,L\rbrack\times_h \Sigma$ be a Riemannian manifold equipped with the metric $g=dt^2+h(t)^2g_\Sigma$, where $(\Sigma,g_\Sigma)$ is a closed Riemannian manifold of dimension 2 and $h$ satisfies condition $(H)$. Then, for all $j\geq1$, we have
    \[\sup\limits_h\{\sigma_j(h):h \text{ satisfies }(H)\}=\frac{L\lambda_j}{2}\,.\]
\end{theoreme}

The proof follows from an intermediary result in which a family of smooth functions $(h_\varepsilon)$ such that $h_\varepsilon(t)\to\infty$ for almost every $t\in\lbrack0,L\rbrack$ and $\sigma_j(h_\varepsilon)\to\frac{\lambda_j L}{2}$ as $\varepsilon\to0$ is constructed, see Section \ref{sec:proofs}. Furthermore, the corresponding eigenfunctions $u_j^\varepsilon$ are uniformly controlled. In Remark \ref{rem:nolower}, we observe that there do not exist a non trivial lower bound.

In Section \ref{sec:stability}, we give a condition that a function $h$ must satisfy in order for $\sigma_j(h)$ to be close to the supremum.

In Section \ref{sec:revolution}, we discuss the case of metrics of revolution on $\B^3$ that are studied in \cite{X21,X22} and recently in \cite{BrCoGi2024}. The authors consider metrics of revolution on $n$-dimensional balls, that is metrics of the form $g_h=dt^2+h(t)^2g_{\s^{n-1}}$ on $\lbrack0,L\rbrack\times\s^{n-1}$. They denote the distinct Steklov eigenvalues, where multiplicity is not counted, by
\[0=\sigma_{(0)}(h)<\sigma_{(1)}(h)<\sigma_{(2)}(h)<\ldots \nearrow\infty\,.\]
The multiplicity of $\sigma_{(j)}(h)$ is  the multiplicity of the $j$-th distinct Laplace eigenvalue $\lambda_{(j)}$ of the sphere $\s^{n-1}$.

In \cite{X22}, under various assumptions on the Ricci curvature of $g_h$ and on the convexity of the boundary, the author produces sharp lower and upper bounds for $\sigma_{(j)}(h)$, see \cite[Theorems 2 and 4]{X22} respectively. Moreover, the equality case corresponds to the ball.
In dimension 3, the metrics considered by the author are very close to the ones studied in this paper. The differences are the following:
\begin{itemize}
    \item The center of the ball corresponds to $t=L$.
    \item We impose that $h(L)=0$ and $h'(L)=-1$.
    \item The value of $h$ at $t=0$ is not fixed.
\end{itemize}
We obtain sharp upper bounds for the distinct Steklov eigenvalues without any assumptions on the curvature of the warped product metric or on the convexity of the boundary.
\begin{theoreme}\label{thm:revolution}
    For a metric $g_h=dt^2+h(t)^2g_{\s^{2}}$ on $\lbrack0,L\rbrack\times\s^{2}$, where $h$ is a smooth function satisfying $h(L)=0$ and $h'(L)=-1$ and $\vert \s^{2}\vert=1$, we have
    \[\sigma_{(j)}(h)<\frac{L\lambda_{(j)}}{h(0)^2}\,.\]
    Moreover, if we fix the value of $h$ at $t=0$, the bound is sharp. Namely, we have
    \[\sup\limits_h\{\sigma_{(j)}(h):h(0)=h_0\}=\frac{L\lambda_{(j)}}{h_0^2}=\frac{Lj(j+1)}{h_0^2}\,.\]
\end{theoreme}

In \cite{X21}, the author studies the eigenvalue gaps and ratios for metrics of revolution on the ball. Under assumptions on the Ricci curvature and on the convexity of the boundary, the author establishes sharp lower and upper bounds for the Steklov spectral gaps and Steklov ratios, see \cite[Theorems 2 and 5]{X21}. In \cite{BrCoGi2024}, the authors obtain results for the Steklov ratios without any assumption on the curvature or the convexity of the boundary, but show that, if $n>2$, no upper bound exists for the gap.

As a corollary of Theorem \ref{thm:revolution}, we get an optimal upper bound for the gap $\sigma_{(j+1)}(h)-\sigma_{(j)}(h)$ without any assumptions on the curvature or on the convexity of the boundary.
\begin{theoreme} \label{thm:gaprevolution}
    For a metric $g_h=dt^2+h(t)^2g_{\s^{2}}$ on $\lbrack0,L\rbrack\times\s^{2}$, where $h$ is a smooth function satisfying $h(L)=0$ and $h'(L)=-1$, we have for each $j\geq0$
    \[\sigma_{(j+1)}(h)-\sigma_{(j)}(h) < \frac{L(\lambda_{(j+1)}-\lambda_{(j)})}{h(0)^2}\,.\]
    Moreover, if we fix the value of $h$ at $t=0$, the upper bound is optimal:
    \[\sup\limits_h\{\sigma_{(j+1)}(h)-\sigma_{(j)}(h):h(0)=h_0\}=\frac{L(\lambda_{(j+1)}-\lambda_{(j)})}{h_0^2}\,.\]
\end{theoreme}

For further information about the gap and ratio problem in the case of the Schrödinger operator, the reader may look at \cite[Chapters 6, 7]{BeLiLo2012}, in the case of the Robin problem at \cite{lau19} and in the case of the Steklov problem at \cite[Chapter 4]{CGGS2024} and at the introduction of \cite{BrCoGi2024}.

\section{Proofs of Theorems \ref{thm:upperbound} and \ref{thm:supev}}\label{sec:proofs}

We start this section with a discussion on the Steklov problem on $M=\lbrack0,L\rbrack\times_h\Sigma$.
Let $\{\phi_j\}$ be an orthonormal basis of eigenfunctions of the Laplacian on $\Sigma$. Denote by $\lambda_j$ the eigenvalue associated to $\phi_j$.
By separation of variables, the Steklov eigenfunctions are of the form $u=a_j\phi_j$ where $a_j$ is a solution of the differential equation
\begin{equation}\label{eq:dea}
    (h^2 a')'-\lambda_ja=0\,.
\end{equation}
Thus, the Rayleigh quotient of $u$ is given by
\[R(u)=\frac{\int\limits_{0}^La_j'(t)^2h(t)^2+\lambda_ja_j(t)^2\,dt}{a_j(0)^2+a_j(L)^2}\,.\]
By solving the differential equation \eqref{eq:dea} with the Steklov boundary condition, we obtain, for each $j\geq0$ fixed, two eigenvalues $\sigma_{j,1}(h)\leq\sigma_{j,2}(h)$. We respectively denote the associated eigenfunctions by $a_{j,1}$ and $a_{j,2}$.
By defining 
$$R_{j,h}(a):=\int\limits_{0}^La'(t)^2h(t)^2+\lambda_ja(t)^2\,dt\,,$$
the previous eigenvalues admit the variational characterization:
\begin{gather}
    \sigma_{j,1}(h)=\min\bigg\{R_{j,h}(a):a(0)^2+a(L)^2=1\bigg\}=R_{j,h}(a_{j,1})\label{eq:cara1}\,,\\
    \sigma_{j,2}(h)=\min\bigg\{R_{j,h}(a):a(0)a_{j,1}(0)+a(L)a_{j,1}(L)=0\,,a(0)^2+a(L)^2=1\bigg\}=R_{j,h}(a_{j,2}) \label{eq:cara2}\,.
\end{gather}
The Steklov spectrum of $M$ is the union of each family, i.e. $\cup_j\{\sigma_{j,1}(h),\sigma_{j,2}(h)\}$. In general, we can say that $\sigma_j(h)\leq\sigma_{j,1}(h)$, but the equality case is not guaranteed.

\medskip
A simple example is the cylinder $\lbrack 0,L\rbrack\times\Sigma$.
\begin{exemple}\label{ex:cylinder}
    Consider the case $h\equiv1$. In that case, the functions $a_{j,1}$ and $a_{j,2}$ satisfy the problem
    \begin{equation*}
    \begin{cases}
        a''-\lambda_j a=0\,,\\
        a'(L)=\sigma a(L)\,,\\
        a'(0)=-\sigma a(0)\,.
    \end{cases}
    \end{equation*}
    The solution to the differential equation is $a(r)=A\sinh(\sqrt{\lambda_j}r)+B\cosh(\sqrt{\lambda_j}r)$ if $j\neq0$.
    The boundary conditions become
    \begin{gather*}
        A(\sqrt{\lambda_j}\cosh(\sqrt{\lambda_j}L)-\sigma\sinh(\sqrt{\lambda_j}L))+B(\sqrt{\lambda_j}\sinh(\sqrt{\lambda_j}L)-\sigma\cosh(\sqrt{\lambda_j}L))=0\\
        A\sqrt{\lambda_j}+B\sigma=0\,.
    \end{gather*}
    We have a nonzero solution for the constants $A$ and $B$ if, and only if, the determinant of the system is 0, i.e if
    \[\sinh(\sqrt{\lambda_j}L)\sigma^2-2\sqrt{\lambda_j}\cosh(\sqrt{\lambda_j}L)\sigma+\lambda_j\sinh(\sqrt{\lambda_j}L)=0\,,\]
    which is verified when $\sigma=\sqrt{\lambda_j}\tanh\big(\frac{\sqrt{\lambda_j}L}{2}\big)$ and $\sigma=\sqrt{\lambda_j}\coth\big(\frac{\sqrt{\lambda_j}L}{2}\big)$.
    Thus, we have
    \[\sigma_{j,1}=\sqrt{\lambda_j}\tanh\bigg(\frac{\sqrt{\lambda_j}L}{2}\bigg)<\sqrt{\lambda_j}\coth\bigg(\frac{\sqrt{\lambda_j}L}{2}\bigg)=\sigma_{j,2}\,.\]
    If $j=0$, then the solution to the problem is $a(r)=A+Br$. The boundary conditions become
    \begin{gather*}
        A\sigma+B(\sigma L-1)=0\\
        A\sigma +B=0\,.
    \end{gather*}
    We have a nonzero solution for the constants $A$ and $B$ if, and only if, the determinant of the system is 0, i.e. if
    \[\sigma^2 L-2\sigma=0\,,\]
    which is verified for $\sigma=0$ and $\sigma=\frac{2}{L}$. Thus, we have $\sigma_{0,1}=0<\frac{2}{L}=\sigma_{0,2}$.
\end{exemple}

\begin{proof}[Proof of Theorem \ref{thm:upperbound}]
    First of all, by taking the constant function in the characterization \eqref{eq:cara1}, we obtain that
    \[\sigma_j(h)\leq\sigma_{j,1}(h)\leq \frac{\lambda_j L}{2}\,.\]
To obtain the strict inequality, observe that if there exists $h\in C^\infty(\lbrack 0,L\rbrack)$ such that $h(L)=h(0)=1$ and $\sigma_j(h)=\frac{L\lambda_j}{2}$, then we can construct another function $\overline{h}\in C^\infty(\lbrack 0,L\rbrack)$ such that $\overline{h}(L)=\overline{h}(0)=1$ and $\overline{h}(r)>h(r)$ for all $r\in (0,L)$. 
Let $\overline{a}_{j,1}$ be the eigenfunction for $\sigma_{j,1}(\overline{h})$.
Let us observe first that $\overline{a}_{j,1}$ is not the constant function since the constant function is not a solution of Equation \eqref{eq:dea} when $j\geq1$.
By taking $\overline{a}_{j,1}$ as a test function for $\sigma_{j,1}(h)$, we obtain that
\[\sigma_{j,1}(h)\leq \int\limits_{0}^L \overline{a}_{j,1}'(t)^2h(t)^2+\lambda_j\overline{a}_{j,1}(t)^2\,dt<\int\limits_{0}^L \overline{a}_{j,1}'(t)^2\overline{h}(t)^2+\lambda_j\overline{a}_{j,1}(t)^2\,dt=\sigma_{j,1}(\overline{h})\leq \frac{L\lambda_j}{2}\,,\]
which is a contradiction.
\end{proof}

Theorem \ref{thm:supev} is a consequence of Theorem \ref{thm:upperbound} and the following result.
\begin{proposition}\label{prop:supev}
     Let $(\Sigma,g_\Sigma)$ be a closed Riemannian manifold of dimension 2 and consider $M=\lbrack 0,L\rbrack\times \Sigma$ equipped with the Riemannian metric $g=dt^2+g_\Sigma$. 
     For each $k\geq1$ fixed, there exist $\varepsilon_0>0$ and a family of Riemannian metrics $g_\varepsilon=dt^2+h_\varepsilon(t)^2g_\Sigma$, defined for $\varepsilon<\varepsilon_0$, which coincides with $g$ on a neighbourhood of the boundary $\partial M$, such that, for $1\leq j\leq k$, we have
     \[\lim\limits_{\varepsilon\to0}\sigma_j(h_\varepsilon)=\frac{\lambda_j L}{2}\,.\]
     Moreover, if $a_j\phi_j$ is an eigenfunction associated to $\sigma_j(h_\varepsilon)$, then $a_j$ converges uniformly to $1$ as $\varepsilon\to0$.
\end{proposition}

\begin{proof}
Let $h_\varepsilon$ be a smooth function such that
$$
h_\varepsilon(t):=\begin{cases}
    1\,,&\mbox{ if $t\in\lbrack0,\varepsilon\rbrack\cup\lbrack L-\varepsilon,L\rbrack$,}\\
    \frac{1}{\varepsilon}\,,&\mbox{ if $t\in\lbrack 2\varepsilon,L-2\varepsilon\rbrack$,}
\end{cases}
$$
and $h_\varepsilon$ increases on $\lbrack \varepsilon,2\varepsilon\rbrack$ and decreases on $\lbrack L-2\varepsilon,L-\varepsilon\rbrack$.

\medskip
Consider the metric 
$g_\varepsilon:=dt^2+h_\varepsilon^2(t)g_\Sigma$ on $M$.

\medskip
Let $a$ be a function defined on $\lbrack 0,L\rbrack$. We show that $|a(L)-a(0)|$ is controlled by the first term of its Rayleigh quotient.
On $\lbrack 0,2\varepsilon\rbrack$, since $h_\varepsilon\geq1$, we have
\[\int\limits_{0}^{2\varepsilon} a'(t)^2h_\varepsilon^2(t)\,dt\geq\int\limits_{0}^{2\varepsilon} a'(t)^2\,dt\,.\]
By the Cauchy--Schwarz inequality, it follows that
\[|a(2\varepsilon)-a(0)|=\bigg|\int\limits_{0}^{2\varepsilon}a'(t)\,dt\bigg|\leq \sqrt{2\varepsilon}\bigg(\int\limits_{0}^{2\varepsilon} a'(t)^2\,dt\bigg)^{1/2}\,,\]
which implies that
\[|a(2\varepsilon)-a(0)|\leq\sqrt{2\varepsilon}\bigg(\int\limits_{0}^{2\varepsilon} a'(t)^2h_\varepsilon^2(t)\,dt\bigg)^{1/2}\,.\]

Similarly, on $\lbrack L-2\varepsilon,L\rbrack$, we have
\[|a(L)-a(L-2\varepsilon)|\leq\sqrt{2\varepsilon}\bigg(\int\limits_{L-2\varepsilon}^{L} a'(t)^2h_\varepsilon^2(t)\,dt\bigg)^{1/2}\,.\]

For each $r\in\lbrack 2\varepsilon,L-2\varepsilon\rbrack$, since $h_\varepsilon=\frac{1}{\varepsilon}$, we have
\[\int\limits_{2\varepsilon}^{r} a'(t)^2h_\varepsilon^2(t)\,dt=\frac{1}{\varepsilon^2}\int\limits_{2\varepsilon}^{r} a'(t)^2\,dt\,.\]
By the Cauchy--Schwarz inequality, it follows that
\[|a(r)-a(2\varepsilon)|\leq \sqrt{r-2\varepsilon}\bigg(\int\limits_{2\varepsilon}^{r} a'(t)^2\,dt\bigg)^{1/2}\,,\]
which implies that
\[|a(r)-a(2\varepsilon)|\leq\sqrt{L}\varepsilon\bigg(\int\limits_{2\varepsilon}^{r} a'(t)^2h_\varepsilon^2(t)\,dt\bigg)^{1/2}\,.\]

Thus, for each $r\in\lbrack 0,L\rbrack$, we have that
\begin{equation}\label{eq:diffa1}
|a(r)-a(0)|\leq C\sqrt{\varepsilon}\bigg(\int\limits_{0}^L a'(t)^2h_\varepsilon^2(t)\,dt\bigg)^{1/2}\,.
\end{equation}

It is clear that $\sigma_{0,1}(h_\varepsilon)=0$ with $a_{0,1}\equiv1$. By the characterization \eqref{eq:cara2} of $\sigma_{0,2}(h_\varepsilon)$, the function $a_{0,2}$ satisfies $a_{0,2}(0)=-a_{0,2}(L)$. Thus, by the Inequality \eqref{eq:diffa1} with $r=L$, we conclude that
\[\lim\limits_{\varepsilon\to0}\sigma_{0,2}(h_\varepsilon)=+\infty\,.\]
For $j\geq1$, in order to have an eigenfunction $a_{j,1}$ for $\sigma_{j,1}(h_\varepsilon)$, it must minimize the Inequality \eqref{eq:diffa1} with $r=L$, which implies that $|a_{j,1}(L)-a_{j,1}(0)|\to 0$ as $\varepsilon\to 0$.
Otherwise, if $|a(L)-a(0)|\not\to0$ as $\varepsilon\to0$, then the Inequality \eqref{eq:diffa1} with $r=L$ gives us that
\[\lim\limits_{\varepsilon\to0}R_{j,h_\varepsilon}(a)=+\infty\,.\]
Thus, for $k\geq1$ fixed, there exists an $\varepsilon_0>0$ such that for every $\varepsilon<\varepsilon_0$ and every $0\leq j\leq k$, we have
\[\sigma_{j,2}(h_\varepsilon)>\frac{L\lambda_j}{2}\,.\]
We deduce that 
$$\sigma_j(h_\varepsilon)=\sigma_{j,1}(h_\varepsilon)$$ 
for all $0\leq j\leq k$.

Let $1\leq j\leq k$. We already know from the proof of Theorem \ref{thm:upperbound}, that $\sigma_{j,1}(h_\varepsilon)\leq \frac{L\lambda_j}{2}$, which implies that $\int\limits_{0}^L a_{j,1}'(t)^2h_\varepsilon(t)^2\,dt\leq \sigma_{j,1}(h_\varepsilon)\leq\frac{L\lambda_j}{2}$. Without loss of generality, we assume that $a_{j,1}(0)>0$. Then, it follows from Equation \eqref{eq:diffa1} that, for all $r\in\lbrack0,L\rbrack$, we have
\[|a_{j,1}(r)-a_{j,1}(0)| \leq C_1 \sqrt{\varepsilon}\,.\]
We then have that
\[a_{j,1}(r)^2\geq (a_{j,1}(0)-C_1\sqrt{\varepsilon})^2\geq a_{j,1}(0)^2-2a_{j,1}(0)C_1\sqrt{\varepsilon}\,,\]
and that
\[a_{j,1}(r)^2\leq (a_{j,1}(0)+C_1\sqrt{\varepsilon})^2\leq a_{j,1}(0)^2+C_2\sqrt{\varepsilon}\,.\]
Thus, we have
\begin{gather*}
\sigma_{j,1}(h_\varepsilon)\geq \frac{\int\limits_{0}^L\lambda_j a_{j,1}(t)^2\,dt}{a_{j,1}(0)^2+a_{j,1}(L)^2}\geq \frac{L\lambda_ja_{j,1}(0)^2-C_3\sqrt{\varepsilon}}{2a_{j,1}(0)^2+C_2\sqrt{\varepsilon}}=\bigg(\frac{\lambda_j L}{2}-C_4\sqrt{\varepsilon}\bigg)\frac{1}{1+C_5\sqrt{\varepsilon}}\\
=\bigg(\frac{L\lambda_j}{2}-C_4\sqrt{\varepsilon}\bigg)\bigg(1-C_5\sqrt{\varepsilon}+O(\varepsilon)\bigg)=\frac{\lambda_j L}{2}-C_6\sqrt{\varepsilon}+O(\varepsilon)\,.
\end{gather*}
Taking the limit as $\varepsilon\to0$ concludes the proof.
Moreover, for $1\leq j\leq k$, the function $a_{j,1}$ converges uniformly to the constant function, which can be chosen to be 1, by Equation \eqref{eq:diffa1}.
\end{proof}

\begin{remarque}  \label{rem:nolower}

    It is natural to also ask if there exists a lower bound for $\sigma_j(h)$. In the case studied here, there is no non trivial lower bound. Indeed, define a family of smooth functions $(h_\varepsilon)$ by
    $$
    h_\varepsilon(t):=\begin{cases}
        1\,,&0\leq t\leq \varepsilon\,,L-\varepsilon\leq t\leq L\,,\\
        \varepsilon^2\,,&2\varepsilon\leq t\leq L-2\varepsilon\,.
    \end{cases}
     $$
     and $h_\varepsilon$ decreases on $\lbrack \varepsilon,2\varepsilon\rbrack$ and increases on $\lbrack L-2\varepsilon,L-\varepsilon\rbrack$.

     \medskip
    Consider the function 
    $$
    a(t):=\begin{cases}
    1\,,&0\leq t\leq 2\varepsilon\,,\\
    3-\frac{t}{\varepsilon}\,,&2\varepsilon\leq t\leq 3\varepsilon\,,\\
    0\,,&3\varepsilon\leq t\leq L\,.
    \end{cases}
    $$
    By using $a$ as a test function in the variational characterization \eqref{eq:cara1} for $\sigma_{j,1}(h_\varepsilon)$, we obtain that
    \begin{gather*}
        \sigma_j(h_\varepsilon)\leq\sigma_{j,1}(h_\varepsilon)\leq\int\limits_0^{3\varepsilon} a'(t)^2h(t)^2+\lambda_ja(t)^2\,dt\leq \int\limits_{2\varepsilon}^{3\varepsilon}\varepsilon^2\,dt+\int\limits_0^{3\varepsilon}\lambda_j\,dt=3\lambda_j\varepsilon+\varepsilon^3\,.
    \end{gather*}
    Thus, we have that $\lim\limits_{\varepsilon\to0}\sigma_j(h_\varepsilon)=0$.
\end{remarque}

\section{Approaching the supremum}\label{sec:stability}

Since there is no function $h\in C^\infty(\Omega)$ such that $h(0)=h(L)=1$ for which we have $\sigma_j(h)= \frac{L\lambda_j}{2}$, it is natural to investigate what conditions must a function $h$ satisfy in order for $\sigma_j(h)$ to be close to the supremum in Theorem \ref{thm:supev}.
We show that if $h$ is bounded above on a small interval, then $\sigma_j(h)$ is far from the supremum.

\begin{theoreme}\label{thm:stabint}
    Suppose that there exist $0<L_1<L_2<L$ and a positive constant $c>0$ such that $h\leq c$ on $\lbrack L_1,L_2\rbrack$. Then, for each $j\geq1$, we have
    \[\sigma_k(h)\leq\frac{\lambda_j L}{2}-\gamma\,,\]
    where $\gamma=\min\bigg\{\frac{\lambda_j(L_2-L_1)}{4}\,,\frac{3\lambda_j^2(L_2-L_1)^3}{8(12c^2+\lambda_j(L_2-L_1)^2)}\bigg\}$.
\end{theoreme}

\begin{remarque}
    In particular, when $c\to\infty$ or $L_2-L_1\to0$, we have that $\gamma\to0$.
    Moreover, if $c\geq\frac{\sqrt{\lambda_j}(L_2-L_1)}{2\sqrt{6}}$, then we choose $\gamma=\frac{3\lambda_j^2(L_2-L_1)^3}{8(12c^2+\lambda_j(L_2-L_1)^2)}$ and if $c\leq\frac{\sqrt{\lambda_j}(L_2-L_1)}{2\sqrt{6}}$, then we choose $\gamma=\frac{\lambda_j(L_2-L_1)}{4}$.
\end{remarque}

\begin{proof}
    For $\delta>0$ to be chosen appropriately, we construct
    $$a(r)=\begin{cases}
        1\,,&0\leq r\leq L_1\,,\\
        1-\delta(r-L_1)\,,& L_1\leq r\leq\frac{L_1+L_2}{2}\,,\\
        1+\delta(r-L_2)\,,&\frac{L_1+L_2}{2}\leq r\leq L_2\,,\\
        1\,,&L_2\leq r\leq L.
    \end{cases}
    $$
    In order to have a positive function, we need $\delta<\frac{2}{L_2-L_1}$.
    We have that the Rayleigh quotient of $a$ is given by
    \[R_{j,h}(a)=\frac{\int\limits_0^La'(r)^2h^2(r)+\lambda_ja(r)^2\,dr}{2}=\frac{\lambda_j L}{2}+\frac{\delta^2\lambda_j(L_2-L_1)^3}{24}-\frac{\delta \lambda_j(L_2-L_1)^2}{4}+\frac{1}{2}\int\limits_{L_1}^{L_2} \delta^2h(r)^2\,dr\,.\]
    Since $h\leq c$ on $\lbrack L_1,L_2\rbrack$, we have that
    \[R_{j,h}(a)\leq \frac{\lambda_j L}{2}+\frac{\delta(L_2-L_1)}{2}\bigg(\delta\bigg(c^2+\frac{\lambda_j(L_2-L_1)^2}{12}\bigg)-\frac{\lambda_j(L_2-L_1)}{2}\bigg)\,.\]
If $\delta\leq\frac{3\lambda_j(L_2-L_1)}{12c^2+\lambda_j(L_2-L_1)^2}$, then we have that
\[\delta\bigg(c^2+\frac{\lambda_j(L_2-L_1)^2}{12}\bigg)-\frac{\lambda_j(L_2-L_1)}{2}\leq\frac{-\lambda_j(L_2-L_1)}{4}\,,\]
which implies that
\[R_{j,h}(a)\leq\frac{\lambda_j L}{2}-\frac{\delta\lambda_j(L_2-L_1)^2}{8}\,.\]
    Finally, by setting $\delta=\min\bigg\{\frac{2}{L_2-L_1}\,,\frac{3\lambda_j(L_2-L_1)}{12c^2+\lambda_j(L_2-L_1)^2}\bigg\}$, we have, by the variational characterization \eqref{eq:cara1} of $\sigma_{j,1}(h)$ and by the fact that $\sigma_j(h)\leq\sigma_{j,1}(h)$, that
    \[\sigma_j(h)\leq\frac{\lambda_j L}{2}-\gamma\,,\]
    where $\gamma=\min\bigg\{\frac{\lambda_j(L_2-L_1)}{4}\,,\frac{3\lambda_j^2(L_2-L_1)^3}{8(12c^2+\lambda_j(L_2-L_1)^2)}\bigg\}$.
\end{proof}

\section{Metrics of revolution}\label{sec:revolution}

In this section, we apply the method developed in Section \ref{sec:proofs} to prove Theorem \ref{thm:revolution}. The difference is that, after separation of variables, the functions $a_j$ must satisfy $a_j(L)=0$ for each $j\geq1$ in addition to satisfying the differential equation \eqref{eq:dea}. Thus, this problem admits a unique solution that has the associated eigenvalue $\sigma_{(j)}(h)$.

\begin{proof}[Proof of Theorem \ref{thm:revolution}]
    Just like in the proof of Theorem \ref{thm:upperbound}, we use a test function in the variational characterization of $\sigma_{(j)}(g_h)$. The only difference is that we cannot use the constant function since we need to have $a(L)=0$.

    As $h$ must satisfy $h(L)=0$ and $h'(L)=-1$, there exists $\varepsilon>0$ such that for $t\in (L-\varepsilon,L\rbrack$, we have
    \[\frac{L-t}{2}\leq h(t)\leq 2(L-t)\,.\]
    We introduce the function
    $$
    a(t)=\begin{cases}
        1\,,&0\leq t\leq L-\varepsilon\,,\\
        \frac{L-t}{\varepsilon}\,,&L-\varepsilon\leq t\leq L\,.
    \end{cases}
    $$
    Then, its Rayleigh quotient satisfies
    \begin{gather*}
        R_{j,h}(a)=\frac{1}{h(0)^2}\bigg(\lambda_{(j)}(L-\varepsilon)+\int\limits_{L-\varepsilon}^{L}\frac{h(t)^2}{\varepsilon^2}+\lambda_{(j)}\bigg(\frac{L-t}{\varepsilon}\bigg)^2\,dt\bigg)\\
        \leq \frac{1}{h(0)^2}\bigg(\lambda_{(j)}(L-\varepsilon)+\frac{(4+\lambda_{(j)})\varepsilon}{3}\bigg)
    \end{gather*}
    By taking the limit as $\varepsilon\to0$, we prove that
    \[\sigma_{(j)}(h)\leq\frac{\lambda_{(j)}L}{h(0)^2}\,.\]
    Moreover, the inequality is strict by the same argument used in the proof of Theorem \ref{thm:upperbound}.

    To show that $\sup\{\sigma_{(j)}(h):h(0)=h_0\}=\frac{L\lambda_{(j)}}{h_0^2}$, we copy verbatim the proof of Proposition \ref{prop:supev} but with the family of smooth functions $(h_\varepsilon)$ defined by
    $$
    h_\varepsilon(t)=\begin{cases}
        h_0\,,&0\leq t\leq\varepsilon\,,\\
        \frac{h_0}{\varepsilon}\,,&2\varepsilon\leq t\leq L-2\varepsilon\,,\\
        L-t\,,&L-\varepsilon\leq t\leq L\,.
    \end{cases}
    $$
    and $h_\varepsilon$ increases on $\lbrack \varepsilon,2\varepsilon\rbrack$ and decreases on $\lbrack L-2\varepsilon,L-\varepsilon\rbrack$.
\end{proof}

\begin{proof}[Proof of Theorem \ref{thm:gaprevolution}]
    Let $a_j$ be the eigenfunction associated to $\sigma_{(j)}(h)$. By taking it as a test function in the variational characterization of $\sigma_{(j+1)}(h)$, we get
    \begin{gather*}
        \sigma_{(j+1)}(h)\leq \frac{1}{h^2(0)}\int\limits_0^La_j'(t)^2h(t)^2+\lambda_{(j+1)}a_j(t)^2\,dt\\
        =\frac{1}{h^2(0)}\int\limits_0^La_j'(t)^2h(t)^2+\lambda_{(j)}a_j(t)^2\,dt+\frac{1}{h^2(0)}\int\limits_0^L (\lambda_{(j+1)}-\lambda_{(j)})a_j(t)^2\,dt\\
        =\sigma_{(j)}(h)+\frac{\lambda_{(j+1)}-\lambda_{(j)}}{\lambda_{(j)}}\frac{1}{h^2(0)}\int\limits_0^L \lambda_{(j)}a_j(t)^2\,dt<\sigma_{(j)}(h)+\frac{(\lambda_{(j+1)}-\lambda_{(j)})\sigma_{(j)}(h)}{\lambda_{(j)}}\\
        \leq\sigma_{(j)}(h)+\frac{L(\lambda_{(j+1)}-\lambda_{(j)})}{h^2(0)}\,,
    \end{gather*}
    where we use the fact that $\frac{1}{h^2(0)}\int\limits_0^L \lambda_{(j)}a_j(t)^2\,dt<\sigma_{(j)}(h)$ because the eigenfunction $a_j$ is not constant and Theorem \ref{thm:revolution} to obtain the last inequality. Thus, we have that
    \[\sigma_{(j+1)}(h)-\sigma_{(j)}(h)<\frac{L(\lambda_{(j+1)}-\lambda_{(j)})}{h^2(0)}\,.\]
    Moreover, this upper bound is optimal. Indeed, consider the family of smooth functions $(h_\varepsilon)$ constructed in the proof of Theorem \ref{thm:revolution}. By the previous inequality, for $j\geq0$, we have that
    \[\sigma_{(j+1)}(h_\varepsilon)=\sum\limits_{k=0}^{j}\sigma_{(j+1-k)}(h_\varepsilon)-\sigma_{(j-k)}(h_\varepsilon)\leq\sum\limits_{k=0}^j\frac{L(\lambda_{(j+1-k)}-\lambda_{(j-k)})}{h_0^2}=\frac{L\lambda_{(j+1)}}{h_0^2}\,.\]
    By Theorem \ref{thm:revolution}, we know that $\sigma_{(j+1)}(h_\varepsilon)\to\frac{L\lambda_{(j+1)}}{h_0^2}$ as $\varepsilon\to0$. This implies that each term in the previous sum converges, namely, for all $0\leq k\leq j$, we have that
    $$
    \sigma_{(j+1-k)}(h_\varepsilon)-\sigma_{(j-k)}(h_\varepsilon) \to \frac{L(\lambda_{(j+1-k)}-\lambda_{(j-k)})}{h_0^2}\,,
    $$
    as $\varepsilon\to0$.
\end{proof}

\section{Acknowledgements}
The authors acknowledge support of the SNSF project ‘Geometric Spectral Theory’, grant number 200021-19689.

\bibliographystyle{plain}              
\bibliography{bibliographie}

\begin{thebibliography}{1}

\bibitem{BeLiLo2012}
Rafael~D. Benguria, Helmut Linde, and Benjam\'{\i}n Loewe.
\newblock Isoperimetric inequalities for eigenvalues of the {L}aplacian and the {S}chr\"{o}dinger operator.
\newblock {\em Bull. Math. Sci.}, 2(1):1--56, 2012.

\bibitem{BrCoGi2024}
Jade Brisson, Bruno Colbois, and Katie Gittins.
\newblock Spectral ratios for {S}teklov eigenvalues of balls with revolution-type metrics, 2024.

\bibitem{cesg2019}
Bruno Colbois, Ahmad El~Soufi, and Alexandre Girouard.
\newblock Compact manifolds with fixed boundary and large {S}teklov eigenvalues.
\newblock {\em Proc. Amer. Math. Soc.}, 147(9):3813--3827, 2019.

\bibitem{CGGS2024}
Bruno Colbois, Alexandre Girouard, Carolyn Gordon, and David Sher.
\newblock Some recent developments on the {S}teklov eigenvalue problem, 2023.
\newblock arXiv:2212.12528, to appear at Revista Matemática Complutense; published online.

\bibitem{lau19}
Richard~S. Laugesen.
\newblock The {R}obin {L}aplacian---{S}pectral conjectures, rectangular theorems.
\newblock {\em J. Math. Phys.}, 60(12):121507, 31, 2019.

\bibitem{X21}
Changwei Xiong.
\newblock Optimal estimates for {S}teklov eigenvalue gaps and ratios on warped product manifolds.
\newblock {\em Int. Math. Res. Not. IMRN}, (22):16938--16962, 2021.

\bibitem{X22}
Changwei Xiong.
\newblock On the spectra of three {S}teklov eigenvalue problems on warped product manifolds.
\newblock {\em J. Geom. Anal.}, 32(5):Paper No. 153, 35, 2022.

\end{thebibliography}

\end{document}